\documentclass[11pt]{article}

\usepackage{setspace}
\usepackage{amsmath}
\usepackage{amssymb}
\usepackage{amsthm}
\usepackage{multirow}
\usepackage{color}
\usepackage{longtable}
\usepackage{array}
\usepackage{url}
\allowdisplaybreaks[4]

\newtheorem{theorem}{Theorem}[section]

\newtheorem{lemma}[theorem]{Lemma}
\newtheorem{example}[theorem]{Example}

\newtheorem{problem}{Problem}
\newtheorem{conjecture}[problem]{Conjecture}

\newtheorem{remark}{Remark}[section]

\title{A note on permutation polynomials over finite fields}

\author{
Jingxue Ma$^{\text{a}}$ and Gennian Ge$^{\text{b,c,}}$\thanks{The research of G. Ge was supported by the National Natural Science Foundation of China under Grant Nos. 11431003 and 61571310, Beijing Hundreds of Leading Talents Training Project of Science and Technology, and Beijing Municipal Natural Science Foundation.}\\
  \footnotesize $^{\text{a}}$ School of Mathematical Sciences, Zhejiang University, Hangzhou 310027, Zhejiang, China\\
  \footnotesize $^{\text{b}}$ School of Mathematical Sciences, Capital Normal University, Beijing, 100048, China\\
\footnotesize $^{\text{c}}$ Beijing Center for Mathematics and Information
Interdisciplinary Sciences,
Beijing, 100048, China}

\begin{document}

\date{}\maketitle
\begin{abstract}
Permutation polynomials over finite fields constitute an active research area and have applications in many areas of science and engineering. In this paper, two conjectures on permutation polynomials proposed recently by Wu and Li \cite{WuL} are settled. Moreover, a new class of permutation trinomials of the form $x+\gamma \textup{Tr}_{q^n/q}(x^k)$ is also presented, which generalizes two examples of \cite{kyureghyan2016}.

\smallskip

\end{abstract}

\section{Introduction}
Let $\mathbb{F}_{p^{n}}$ be a finite field with $p^{n}$ elements, where $p$ is a prime and $n$ is a positive integer. A polynomial
$f(x)\in\mathbb{F}_{p^{n}}[x]$ is called a permutation polynomial (PP) if the associated mapping $f:c\mapsto f(c)$ from $\mathbb{F}_{p^{n}}$ to itself is a bijection. PPs have been intensively studied in recent years due to their important applications in cryptography, coding theory and combinatorial design theory (see \cite{Ding06,dobb1999,D99,L07,QTTL13,ST05} and the references therein). Recently, the study of PPs with few terms, especially binomials and trinomials, has attached people's interest due to their simple algebraic form and additional extraordinary properties.

In \cite{WuL}, Wu and Li presented several classes of permutation trinomials over $\mathbb{F}_{5^n}$ from Niho exponents with the form of
\begin{equation}\label{eq1.1}
f(x)=x+\lambda_1 x^{s(5^k-1)+1}+\lambda_2 x^{t(5^k-1)+1},
\end{equation}
where $n=2k, 1\leq s,t\leq 5^k,$ and $\lambda_1, \lambda_2\in\{1,-1\}.$ In the same paper, they also proposed the following two conjectures, which can be used to obtain two new classes of permutation trinomials with the form \eqref{eq1.1}. More recent progress on permutation trinomials can be found in \cite{DQWYY15,Gupta2016,Hou2014,Hou2015,H15,kyureghyan2016,LQC2017,Likuanquan2016New,Linian2016,LH2016,Ma2016,Zha2017}.

\begin{conjecture}
The polynomial $f(x)=x\big(\frac{x^2-x+2}{x^2+x+2}\big)^2$ is a PP over $\mathbb{F}_{5^k}$ for odd $k.$
\end{conjecture}

\begin{conjecture}
Let $q=5^k,$ where $k$ is an even integer. Then $g(x)=-x\big(\frac{x^2-2}{x^2+2}\big)^{2}$ permutates $\mu_{q+1}.$
\end{conjecture}

This paper is organized as follows. In Section~\ref{preliminaries}, we introduce some basic notations and related results. In Sections~\ref{conj2} and \ref{conj1}, we prove the above two conjectures by using different tricks, such as treating the squares and non-squares separately. In Section~\ref{trapp}, we give a new class of PPs of the form $x+\gamma \textup{Tr}_{q^n/q}(x^k).$ Section~\ref{conclusion} concludes the paper.

\section{Preliminaries}\label{preliminaries}
The following notations are fixed throughout this paper.
\begin{itemize}
  \item Let $q$ be a prime power, $n$ be an integer, and $\mathbb{F}_{q^n}$ be the finite field of order $q^{n}$.
  \item Let $\textup{Tr}_{r}^{n}\ :\ \mathbb{F}_{q^{n}}\mapsto\mathbb{F}_{q^{r}}$ be the trace mapping defined by
  $$\textup{Tr}_{r}^{n}(x)=x+x^{q^{r}}+x^{q^{2r}}+\cdots+x^{q^{n-r}},$$
  where $r|n$. For $r=1$, we get the absolute trace function, which is denoted by $\textup{Tr}_{n}$.
  \item Let $\overline{x}=x^q$ for any $x\in \mathbb{F}_{q^{2}}$.
\end{itemize}

Now, we recall a well-known lemma which will be needed in the following sections.
\begin{lemma}\cite{Dickson1906}\label{lem2.1}
$x^2+ax+b$ is irreducible over $\mathbb{F}_{p^n},$ $p>2$, if and only if its discriminant $\Delta=a^2-4b$ is a non-square element in $\mathbb{F}_{p^n}.$
\end{lemma}

\section{Proof of Conjecture 2}\label{conj2}
In this section, we will prove the following theorem, which is a conjecture posed by Wu et al. \cite{WuL}. Let us set initially that $\mu_{q+1}=\{x\in\mathbb{F}_{q^2}: x^{q+1}=1\},$ where $q=5^k$ with $k$ being an even integer.
\begin{theorem}\cite[Conjecture 2]{WuL}\label{thm3.1}
Let $q=5^k,$ where $k$ is an even integer. Then $g(x)=-x\big(\frac{x^2-2}{x^2+2}\big)^{2}$ permutates $\mu_{q+1}.$
\end{theorem}

Before proving this conjecture, we need to show the following lemmas.

\begin{lemma}\label{lem0}
Let $q=5^k,$ where $k$ is an even integer. Then $\pm2$ are square elements in $\mathbb{F}_{q^2}.$ Moreover, $\sqrt{\pm2}\in \mathbb{F}_q.$
\end{lemma}
\begin{proof}
Let $\mathbb{F}_{q^2}^*=\langle \omega \rangle.$ Noting that $8|(q^2-1)$ and $-1=\omega^{\frac{q^2-1}{2}}$, we have that $2=\omega^{\frac{q^2-1}{4}}$ and $-2=\omega^{\frac{3(q^2-1)}{4}}$ are two square elements in $\mathbb{F}_{q^2}.$
Write $\sqrt{2}=\omega^{\frac{q^2-1}{8}},$ $\sqrt{-2}=\omega^{\frac{3(q^2-1)}{8}}$. Noting that $8|(q-1)$ since $k$ is an even integer. Thus
$(\sqrt{2})^{q-1}=(-1)^{\frac{q-1}{4}}=1,$ and $(\sqrt{-2})^{q-1}=(-1)^{\frac{3(q-1)}{4}}=1,$ which imply that $\sqrt{2}\in \mathbb{F}_q$, $\sqrt{-2}\in \mathbb{F}_q$.
\end{proof}

Let $\Omega_{+}=\{x^2:x\in\mu_{q+1}\}$, $\Omega_{-}=\{-x^2:x\in\mu_{q+1}\}$.
\begin{lemma}\label{lem1}
$\Omega_{+}\cap \Omega_{-}=\varnothing$, $\Omega_{+}\cup \Omega_{-}=\mu_{q+1}$.
\end{lemma}
\begin{proof}
If $\Omega_{+}\cap \Omega_{-}\neq \varnothing,$ that is $\exists x_1^2=-x_2^2$ with $x_i\in \mu_{q+1}$. We have $\big( \frac{x_1}{x_2}\big)^2=-1$, which means $\big( \frac{x_1}{x_2}\big)^4=1$. Since $\big( \frac{x_1}{x_2}\big)^{q+1}=1$, it follows that $\big( \frac{x_1}{x_2}\big)^{\gcd(4, q+1)}=1$. Hence, $\big( \frac{x_1}{x_2}\big)^2=1$, this leads to a contradiction.

Clearly, $|\Omega_{+}|=|\Omega_{-}|=\frac{q+1}{2}$. It follows that $\Omega_{+}\cup \Omega_{-}=\mu_{q+1}$.
\end{proof}

\begin{lemma}\label{lem2}
$g(\Omega_{+})\subseteq \Omega_{+}$ and $g(\Omega_{-})\subseteq \Omega_{-}$.
\end{lemma}
\begin{proof}
$\forall x\in \Omega_{+}$, $\exists a\in \mu_{q+1}$, s.t. $x=a^2$. We need to show $g(x)=b^2$ for some $b\in \mu_{q+1}$. In fact, $g(x)=g(a^2)=-a^2\big( \frac{a^4-2}{a^4+2}\big)^2=\Big(2a\big(\frac{a^4-2}{a^4+2}\big)\Big)^2$. Let $b=2a\big(\frac{a^4-2}{a^4+2}\big)$, noting that $\overline{a}=\frac{1}{a},$ we have
\begin{align*}
\overline{b}=2\overline{a}\big(\frac{\overline{a}^4-2}{\overline{a}^4+2}\big)=\frac{2}{a}\Big( \frac{(\frac{1}{a})^4-2}{(\frac{1}{a})^4+2} \Big)=\frac{1}{2a}\big(\frac{a^4+2}{a^4-2}\big)=\frac{1}{b}.
\end{align*}

Similarly, we have $g(\Omega_{-})\subseteq \Omega_{-}$. In fact, $\forall x\in \Omega_{-}$, $\exists a\in \mu_{q+1}$, s.t. $x=-a^2$. Since $g(x)=g(-a^2)=a^2\big( \frac{a^4-2}{a^4+2}\big)^2=-\Big(2a\big(\frac{a^4-2}{a^4+2}\big)\Big)^2$. Let $b=2a\big(\frac{a^4-2}{a^4+2}\big)$, noting that $b\in \mu_{q+1},$ we have shown $g(x)\in\Omega_{-}.$
\end{proof}

\begin{lemma}\label{lem3.1}
The group of equations
\begin{equation*}
\begin{cases}
xy=1,\\
x+y=\pm1
\end{cases}
\end{equation*}
has no solution in $\Omega_{+},$ $\Omega_{-},$ respectively.
\end{lemma}
\begin{proof}
If not, there exist $x,y\in \Omega_{+}$ (or $\Omega_{-}$) such that the equations hold. It follows that $x^2\pm x+1=0$. So $x=\pm2\pm 2\sqrt{2}\in\mathbb{F}_q$ by Lemma \ref{lem0}.
Thus, we obtain that $x^{q-1}=1$. Since $x^{q+1}=1$. It deduces that $x^{\gcd(q-1,q+1)}=1$, that is $x=\pm1$, which is a contradiction.
\end{proof}

\begin{lemma}\label{lem3}
$g(x)$ permutes $\Omega_{+}$.
\end{lemma}

\begin{proof}
If the assertion would not hold, there exist $x\neq y\in\Omega_{+}$, such that $g(x)=g(y)$. Let $x=a^2$, $y=b^2$, where $a,b\in \mu_{q+1}$, and $a\neq \pm b$. Since $g(x)=g(y)$, we have $-a^2\big(\frac{a^4-2}{a^4+2}\big)^2=-b^2\big(\frac{b^4-2}{b^4+2}\big)^2$, which means that $\frac{a^5-2a}{a^4+2}=\pm\frac{b^5-2b}{b^4+2}$.

Below, we consider the two cases where $\frac{a^5-2a}{a^4+2}=\frac{b^5-2b}{b^4+2}$ and $\frac{a^5-2a}{a^4+2}=-\frac{b^5-2b}{b^4+2}$.

{\bf Case 1: $\frac{a^5-2a}{a^4+2}=\frac{b^5-2b}{b^4+2}.$}

We obtain
\begin{equation*}
(a-b)\Big(a^4b^4+2(a-b)^4+2ab(a-b)^2-4a^2b^2-4\Big)=0,
\end{equation*}
which implies that
$$a^4b^4+2(a-b)^4+2ab(a-b)^2-4a^2b^2-4=0,$$
since $a\neq b$.
Let $c=ab$, $d=a-b$, we have
$$c^4+2d^4+2cd^2-4c^2-4=0,$$
that is
$$d^4+cd^2-2(c^4+c^2+1)=0.$$
Thus, we have $d^2=2c\pm 2\sqrt{-2}(c^2-1)$, since $\Delta=-2(c^2-1)^2$ is a square in $\mathbb{F}_{q^2}.$

Now, we consider the following two subcases.

{\bf Subcase 1.1 } For the case
\begin{equation}\label{eq3.1}
d^2=2c+2\sqrt{-2}(c^2-1),
\end{equation}
raising $q$-th power of both sides of Eq. \eqref{eq3.1}, by Lemma \ref{lem0}, we get
\begin{equation*}
\overline{d}^2=2\overline{c}+2\sqrt{-2}(\overline{c}^2-1).
\end{equation*}
Since $\overline{c}=\frac{1}{c}$ and $\overline{d}=\overline{a}-\overline{b}=\frac{1}{a}-\frac{1}{b}=\frac{b-a}{ab}=\frac{-d}{c}$, we obtain
$$\frac{d^2}{c^2}=\frac{2}{c}+2\sqrt{-2}(\frac{1}{c^2}-1),$$
which implies that
\begin{equation}\label{eq3.2}
d^2=2c+2\sqrt{-2}(1-c^2).
\end{equation}
Combining Eq. \eqref{eq3.1} and Eq. \eqref{eq3.2}, we have $c^2=1$.
It follows that
\begin{equation*}
d^2=2c=
\begin{cases}
2,~ \text{if $c=1$,}\\
-2,~ \text{if $c=-1$.}
\end{cases}
\end{equation*}

\begin{enumerate}
\item When $c=1$, $d^2=2$, we have
 \begin{equation}\label{eq0.1}
   \begin{cases}
    xy=a^2b^2=c^2=1,\\
    x+y=a^2+b^2=(a-b)^2+2ab=d^2+2c=-1.
   \end{cases}
 \end{equation}
 By Lemma \ref{lem3.1}, Eq. \eqref{eq0.1} has no solution in $\Omega_{+}$, which is a contradiction.
\item When $c=-1$, $d^2=-2$, we get
 \begin{equation}\label{eq0.2}
  \begin{cases}
   xy=1,\\
   x+y=1.
  \end{cases}
 \end{equation}
 By Lemma \ref{lem3.1}, Eq. \eqref{eq0.2} has no solution in $\Omega_{+}$, which is a contradiction.
\end{enumerate}

{\bf Subcase 1.2} For the case $d^2=2c-2\sqrt{-2}(c^2-1),$ we could also get a contradiction in a similar way. In fact, raising $q$-th power of both sides of $d^2=2c-2\sqrt{-2}(c^2-1)$, by Lemma \ref{lem0}, we get
\begin{equation*}
\overline{d}^2=2\overline{c}-2\sqrt{-2}(\overline{c}^2-1).
\end{equation*}
Noting that $\overline{c}=\frac{1}{c}$ and $\overline{d}=\frac{-d}{c}$, we obtain
$$\frac{d^2}{c^2}=\frac{2}{c}-2\sqrt{-2}(\frac{1}{c^2}-1),$$
which implies that
\begin{equation*}
d^2=2c-2\sqrt{-2}(1-c^2).
\end{equation*}
It follows that $c^2=1.$ The rest of the argument is exactly the same as {\bf Subcase 1.1}.

{\bf Case 2: $\frac{a^5-2a}{a^4+2}=-\frac{b^5-2b}{b^4+2}$}. The discussion is similar with {\bf Case 1}.
In fact, we have
\begin{equation*}
(a+b)\Big(a^4b^4+2(a+b)^4-2ab(a+b)^2-4a^2b^2-4\Big)=0,
\end{equation*}
which implies that
$$a^4b^4+2(a+b)^4-2ab(a+b)^2-4a^2b^2-4=0,$$
since $a\neq -b$.
Let $c=ab$, $d=a+b$, we have
$$c^4+2d^4-2cd^2-4c^2-4=0,$$
that is
$$d^4-cd^2-2(c^4+c^2+1)=0.$$
Therefore, we have $d^2=-2c\pm 2\sqrt{-2}(c^2-1)$, since $\Delta=-2(c^2-1)^2$ is a square in $\mathbb{F}_{q^2}.$

Now, we consider the following two subcases.

{\bf Subcase 2.1 } For the case
\begin{equation}\label{eq3.2.1}
d^2=-2c+2\sqrt{-2}(c^2-1),
\end{equation}
raising $q$-th power of both sides of Eq. \eqref{eq3.2.1}, by Lemma \ref{lem0}, we get
\begin{equation*}
\overline{d}^2=-2\overline{c}+2\sqrt{-2}(\overline{c}^2-1).
\end{equation*}
Since $\overline{c}=\frac{1}{c}$ and $\overline{d}=\overline{a}+\overline{b}=\frac{1}{a}+\frac{1}{b}=\frac{a+b}{ab}=\frac{d}{c}$, we obtain
$$\frac{d^2}{c^2}=\frac{-2}{c}+2\sqrt{-2}(\frac{1}{c^2}-1),$$
which implies that
\begin{equation}\label{eq3.2.2}
d^2=-2c+2\sqrt{-2}(1-c^2).
\end{equation}
Combining Eq. \eqref{eq3.2.1} and Eq. \eqref{eq3.2.2}, we have $c^2=1$.
It follows that
\begin{equation*}
d^2=-2c=
\begin{cases}
-2,~ \text{if $c=1$,}\\
2,~ \text{if $c=-1$.}
\end{cases}
\end{equation*}

\begin{enumerate}
\item If $c=1$, $d^2=-2$, we have
\begin{equation}\label{eq0.3}
\begin{cases}
xy=1,\\
x+y=1.
\end{cases}
\end{equation}
By Lemma \ref{lem3.1}, Eq. \eqref{eq0.3} has no solution in $\Omega_{+}$, which is a contradiction.
\item If $c=-1$, $d^2=2$, we get
\begin{equation}\label{eq0.4}
\begin{cases}
xy=1,\\
x+y=-1.
\end{cases}
\end{equation}
By Lemma \ref{lem3.1}, Eq. \eqref{eq0.4} has no solution in $\Omega_{+}$, which is a contradiction.
\end{enumerate}

{\bf Subcase 2.2} For the case $d^2=-2c-2\sqrt{-2}(c^2-1),$ we could also get a contradiction in a similar way. In fact, raising $q$-th power of both sides of $d^2=-2c-2\sqrt{-2}(c^2-1)$, by Lemma \ref{lem0}, we get
\begin{equation*}
\overline{d}^2=2\overline{c}-2\sqrt{-2}(\overline{c}^2-1).
\end{equation*}
Noting that $\overline{c}=\frac{1}{c}$ and $\overline{d}=\frac{d}{c}$, we obtain
$$\frac{d^2}{c^2}=\frac{2}{c}-2\sqrt{-2}(\frac{1}{c^2}-1),$$
which implies that
\begin{equation*}
d^2=-2c-2\sqrt{-2}(1-c^2).
\end{equation*}
It follows that $c^2=1.$ The rest of the argument is exactly the same as {\bf Subcase 2.1}.
\end{proof}

\begin{lemma}\label{lem4}
$g(x)$ permutes $\Omega_{-}$.
\end{lemma}

\begin{proof}
The proof is similar to that of Lemma \ref{lem3}. In fact, if the assertion would not hold, there exist $x\neq y\in\Omega_{-}$, such that $g(x)=g(y)$. Let $x=-a^2$, $y=-b^2$, where $a,b\in \mu_{q+1}$, and $a\neq \pm b$. Since $g(x)=g(y)$, we have $a^2\big(\frac{a^4-2}{a^4+2}\big)^2=b^2\big(\frac{b^4-2}{b^4+2}\big)^2$, which means that $\frac{a^5-2a}{a^4+2}=\pm\frac{b^5-2b}{b^4+2}$. The rest of the argument is exactly the same as Lemma \ref{lem3}, except for changing the sign of $x+y$ since $x+y=-(a^2+b^2)$, which doesn't affect the rest discussion.
\end{proof}

\begin{proof}[Proof of Theorem \ref{thm3.1}]
It can be derived directly from Lemmas \ref{lem1}, \ref{lem2}, \ref{lem3} and \ref{lem4}.
\end{proof}

\section{Proof of Conjecture 1}\label{conj1}

\begin{theorem}\cite[Conjecture 1]{WuL}\label{thm4.1}
The polynomial $f(x)=x\big(\frac{x^2-x+2}{x^2+x+2}\big)^2$ is a PP over $\mathbb{F}_{5^k}$ for odd $k.$
\end{theorem}
Before proving this conjecture, we need to show the following lemmas. Let $\Omega_{1}=\{x^2:x\in \mathbb{F}_{5^k}^{*}\}$ and $\Omega_{2}=\{2x^2:x\in\mathbb{F}_{5^k}^{*}\}$.

\begin{lemma}\label{lemconj1.1}
$f(x)$ is a PP on $\Omega_{1}.$
\end{lemma}
\begin{proof}
If not, there exist two distinct elements $x,y\in \Omega_1$ such that $f(x)=f(y).$ Let $x=a^2$, $y=b^2,$ where $a,b\in \mathbb{F}_{5^k}^*$, and $a\neq \pm b.$ We have $f(a^2)=f(b^2)$. That is
$$a^2\big( \frac{a^4-a^2+2}{a^4+a^2+2} \big)^2=b^2\big( \frac{b^4-b^2+2}{b^4+b^2+2} \big)^2.$$
We get
$$\frac{a^5-a^3+2a}{a^4+a^2+2}=\pm \frac{b^5-b^3+2b}{b^4+b^2+2}.$$

{\bf Case 1:}
For the case
$$\frac{a^5-a^3+2a}{a^4+a^2+2}=\frac{b^5-b^3+2b}{b^4+b^2+2}.$$

After simple computations, we have
$$(a-b)\big( a^4b^4+2(a-b)^4+(a^2b^2-2ab-2)(a^2+ab+b^2)+a^3b^3-a^2b^2-2ab+4 \big)=0.$$
Let $c=ab$, $d=a-b$. Noting that $a\neq b$, we have
$$c^4+2d^4+(c^2-2c-2)(d^2-2c)+c^3-c^2-2c+4=0.$$
Simplifying the above equation gives
$$d^4-(2c^2+c+1)d^2-2c^4+2c^3-c^2+c+2=0.$$
Let $z=d^2\in\mathbb{F}_{5^k}^*$, we have
\begin{equation}\label{eq4.1}
z^2-(2c^2+c+1)z-2c^4+2c^3-c^2+c+2=0.
\end{equation}

By Lemma \ref{lem2.1}, we know $z^2-(2c^2+c+1)z-2c^4+2c^3-c^2+c+2$ is irreducible over $\mathbb{F}_{5^k}$ if and only if the discriminant $\Delta$ is non-square. It is easy to obtain $\Delta=2(c^2-c-2)^2.$

{\bf Subcase 1.1:}
If $c^2-c-2\neq 0,$ clearly, $\Delta$ is a non-square element in $\mathbb{F}_{5^k}$ since $2$ is non-square. Thus Eq. \eqref{eq4.1} has no solution in $\mathbb{F}_{5^k}$.

{\bf Subcase 1.2:}
If $c^2-c-2=0,$ we have $c=2$ or $c=-1$.

\begin{enumerate}
  \item If $c=2,$ we get $d^2=-2$. Then we have
  \begin{equation}\label{eqn1}
  \begin{cases}
  xy=a^2b^2=c^2=-1, \\
  x+y=a^2+b^2=(a-b)^2+2ab=d^2+2c=2.
  \end{cases}
  \end{equation}
  So $x,y$ are two roots of $u^2-2u-1=0.$ However, $u^2-2u-1$ is irreducible over $\mathbb{F}_{5^k}$ by Lemma \ref{lem0}, since the discriminant $\Delta=3$ is non-square. We get a contradiction.
  \item If $c=-1$, we have $d^2=1$. Then \begin{equation}\label{eqn2}
  \begin{cases}
  xy=1, \\
  x+y=-1.
  \end{cases}
  \end{equation}
  So $x,y$ are two roots of $u^2+u+1=0.$ However, $u^2+u+1$ is irreducible over $\mathbb{F}_{5^k}$ by Lemma \ref{lem0}, since the discriminant $\Delta=2$ is non-square. We get a contradiction.
\end{enumerate}

{\bf Case 2:}
For the case
$$\frac{a^5-a^3+2a}{a^4+a^2+2}=-\frac{b^5-b^3+2b}{b^4+b^2+2}.$$

After simple computations, we have
$$(a+b)\big( a^4b^4+2(a+b)^4+(a^2b^2+2ab-2)(a^2-ab+b^2)-a^3b^3-a^2b^2+2ab+4 \big)=0.$$
Let $c=ab$, $d=a+b$. Noting that $a\neq -b$, we have
$$c^4+2d^4+(c^2+2c-2)(d^2+2c)-c^3-c^2+2c+4=0.$$
Simplifying the above equation gives
$$d^4-(2c^2-c+1)d^2-2c^4-2c^3-c^2-c+2=0.$$
Let $z=d^2\in\mathbb{F}_{5^k}^*$, we have
\begin{equation}\label{eq4.2}
z^2-(2c^2-c+1)z-2c^4-2c^3-c^2-c+2=0.
\end{equation}
By Lemma \ref{lem2.1}, we know $z^2-(2c^2-c+1)z-2c^4-2c^3-c^2-c+2$ is irreducible over $\mathbb{F}_{5^k}$ if and only if the discriminant $\Delta$ is non-square. It is easy to obtain $\Delta=2(c^2+c-2)^2.$

{\bf Subcase 2.1:}
If $c^2+c-2\neq 0,$ clearly, $\Delta$ is a non-square element in $\mathbb{F}_{5^k}$ since $2$ is non-square. Thus Eq. \eqref{eq4.2} has no solution in $\mathbb{F}_{5^k}$.

{\bf Subcase 2.2:}
If $c^2+c-2=0,$ we have $c=-2$ or $c=1$.

\begin{enumerate}
  \item If $c=-2,$ we get $d^2=-2$. Then we have
  \begin{equation}\label{eqn1}
  \begin{cases}
  xy=-1, \\
  x+y=-1.
  \end{cases}
  \end{equation}
  So we get $x=y=2,$ which is a contradiction.
  \item If $c=1$, we have $d^2=1$. Then \begin{equation}\label{eqn2}
  \begin{cases}
  xy=1, \\
  x+y=-1.
  \end{cases}
  \end{equation}
  So $x,y$ are two roots of $u^2+u+1=0.$ However, $u^2+u+1$ is irreducible over $\mathbb{F}_{5^k}$ by Lemma \ref{lem0}, since the discriminant $\Delta=2$ is non-square. We get a contradiction.
\end{enumerate}

\end{proof}

\begin{lemma}\label{lemconj1.2}
$f(x)$ permutates $\Omega_{2}.$
\end{lemma}

\begin{proof}
An argument similar to that of the above Lemma \ref{lemconj1.1} shows that$f(x)$ permutates $\Omega_{2}.$

In fact, if not, there exist two distinct elements $x$ and $y$ such that $f(x)=f(y).$ Let $x=2a^2$, $y=2b^2,$ where $a,b\in \mathbb{F}_{5^k}^*$, and $a\neq \pm b.$ We have $f(2a^2)=f(2b^2)$. That is
$$2a^2\big( \frac{4a^4-2a^2+2}{4a^4+2a^2+2} \big)^2=2b^2\big( \frac{4b^4-2b^2+2}{4b^4+2b^2+2} \big)^2.$$
We get
$$\frac{4a^5-2a^3+2a}{4a^4+2a^2+2}=\pm \frac{4b^5-2b^3+2b}{4b^4+2b^2+2}.$$

Next, we consider the following two cases separately.

{\bf Case 1:}
For the case
$$\frac{4a^5-2a^3+2a}{4a^4+2a^2+2}=\frac{4b^5-2b^3+2b}{4b^4+2b^2+2}.$$

After simple computations, we have
$$(a-b)\big( a^4b^4+3(a-b)^4+(3a^2b^2-3ab-4)(a^2+ab+b^2)+3a^3b^3-4a^2b^2-4ab+4 \big)=0.$$
Let $c=ab$, $d=a-b$. Noting that $a\neq b$, we have
$$c^4+3d^4+(3c^2-3c-4)(d^2-2c)+3c^3-4c^2-4c+4=0.$$
Simplifying the above equation gives
$$d^4+(c^2-c+2)d^2+2c^4-c^3-c^2-2c+3=0.$$
Let $z=d^2\in\mathbb{F}_{5^k}^*$, we have
\begin{equation}\label{eq4.2.1}
z^2+(c^2-c+2)z+2c^4-c^3-c^2-2c-2=0.
\end{equation}

By Lemma \ref{lem2.1}, we know $z^2+(c^2-c+2)z+2c^4-c^3-c^2-2c-2$ is irreducible over $\mathbb{F}_{5^k}$ if and only if the discriminant $\Delta$ is non-square. It is easy to obtain $\Delta=3(c^2+2c+2)^2.$

{\bf Subcase 1.1:}
If $c^2+2c+2\neq 0,$ clearly, $\Delta$ is a non-square element in $\mathbb{F}_{5^k}$ since $2$ is non-square. Thus Eq. \eqref{eq4.2.1} has no solution in $\mathbb{F}_{5^k}$.

{\bf Subcase 1.2:}
If $c^2+2c+2=0,$ we have $c=2$ or $c=1$.

\begin{enumerate}
  \item If $c=2,$ we get $d^2=-2$. Then we have
  \begin{equation}\label{eqn2.1}
  \begin{cases}
  xy=1, \\
  x+y=-1.
  \end{cases}
  \end{equation}
  So $x,y$ are two roots of $u^2+u+1=0.$ However, $u^2+u+1$ is irreducible over $\mathbb{F}_{5^k}$ by Lemma \ref{lem0}, since the discriminant $\Delta=2$ is non-square. We get a contradiction.

  \item If $c=1$, we have $d^2=-1$. Then \begin{equation}\label{eqn2.2}
  \begin{cases}
  xy=-1, \\
  x+y=2.
  \end{cases}
  \end{equation}
  So $x,y$ are two roots of $u^2-2u-1=0.$ However, $u^2-2u-1$ is irreducible over $\mathbb{F}_{5^k}$ by Lemma \ref{lem0}, since the discriminant $\Delta=3$ is non-square. We get a contradiction.

\end{enumerate}

{\bf Case 2:}
For the case
$$\frac{4a^5-2a^3+2a}{4a^4+2a^2+2}=- \frac{4b^5-2b^3+2b}{4b^4+2b^2+2}.$$

After simple computations, we have
$$(a+b)\big( a^4b^4-2(a+b)^4-2(a^2b^2+ab+2)(a^2-ab+b^2)+2a^3b^3-4a^2b^2+4ab+4 \big)=0.$$
Let $c=ab$, $d=a+b$. Noting that $a\neq -b$, we have
$$c^4-2d^4-2(c^2+c+2)(d^2+2c)+2c^3-4c^2+4c+4=0.$$
Simplifying the above equation gives
$$d^4+(c^2+c+2)d^2+2c^4+c^3-c^2+2c-2=0.$$
Let $z=d^2\in\mathbb{F}_{5^k}^*$, we have
\begin{equation}\label{eq4.3.2}
z^2+(c^2+c+2)z+2c^4+c^3-c^2+2c-2=0.
\end{equation}
By Lemma \ref{lem2.1}, we know $z^2+(c^2+c+2)z+2c^4+c^3-c^2+2c-2$ is irreducible over $\mathbb{F}_{5^k}$ if and only if the discriminant $\Delta$ is non-square. It is easy to obtain $\Delta=3(c^2-2c+2)^2.$

{\bf Subcase 2.1:}
If $c^2-2c+2\neq 0,$ clearly, $\Delta$ is a non-square element in $\mathbb{F}_{5^k}$ since $2$ is non-square. Thus Eq. \eqref{eq4.3.2} has no solution in $\mathbb{F}_{5^k}$.

{\bf Subcase 2.2:}
If $c^2-2c+2=0,$ we have $c=-1$ or $c=3$.

\begin{enumerate}
  \item If $c=-1,$ we get $d^2=-1$. Then we have
  \begin{equation}\label{eqn3.1}
  \begin{cases}
  xy=-1, \\
  x+y=2.
  \end{cases}
  \end{equation}
  So $x,y$ are two roots of $u^2-2u-1=0.$ However, $u^2-2u-1$ is irreducible over $\mathbb{F}_{5^k}$ by Lemma \ref{lem0}, since the discriminant $\Delta=3$ is non-square. We get a contradiction.
  \item If $c=3$, we have $d^2=-2$. Then \begin{equation}\label{eqn3.2}
  \begin{cases}
  xy=1, \\
  x+y=-1.
  \end{cases}
  \end{equation}
  So $x,y$ are two roots of $u^2+u+1=0.$ However, $u^2+u+1$ is irreducible over $\mathbb{F}_{5^k}$ by Lemma \ref{lem0}, since the discriminant $\Delta=2$ is non-square. We get a contradiction.
\end{enumerate}

\end{proof}

\begin{proof}[Proof of Theorem \ref{thm4.1}]
Noting that $\Omega_{1}\cap \Omega_{2}=\varnothing$, $\Omega_{1}\cup \Omega_{2}=\mathbb{F}_{5^k}^{*}$. Clearly, $f(\Omega_{1})\subseteq \Omega_{1}$ and $f(\Omega_{2})\subseteq \Omega_{2}.$ It is sufficient to prove that $f(x)$ permutates $\Omega_{1},\Omega_{2},$ respectively. Then, it can be derived directly from Lemmas \ref{lemconj1.1} and \ref{lemconj1.2}.
\end{proof}

\section{PPs of the form $x+\gamma \textup{Tr}_{n}(x^k)$}\label{trapp}

In \cite{kyureghyan2016}, the authors computed all PPs over fields $\mathbb{F}_{q^n}$ of the form $x+\gamma \textup{Tr}_{n}(x^k)$ with $\gamma\in \mathbb{F}_{q^n}^*,$ $n>1$ and $q^n<5000.$ They gave several families of PPs, which explain almost all PPs of this form. However, the following five examples arising in their computation are not covered:
\begin{example}\label{ex5.1}
$q=7,~ n=2, ~k=10, ~\gamma^4=1.$
\end{example}

\begin{example}\label{ex5.2}
$q=9,~ n=2,~ k=33,~ \gamma^2-\gamma=1.$
\end{example}

\begin{example}\label{ex5.3}
$q=27, ~n=2,~ k=261,~ (\gamma-1)^{13}=\gamma^{13}.$
\end{example}

\begin{example}\label{ex5.4}
$q=9,~n=3, ~k=\{11, 19, 33, 57\},~ \gamma^4=-1.$
\end{example}

\begin{example}\label{ex5.5}
$q=49,~ n=2,~ k=385, ~\gamma^5=-1.$
\end{example}

In the following theorem, we generalize Examples \ref{ex5.2} and \ref{ex5.3} into a new infinite class.

\begin{theorem}
Let $q=3^r$ with $r \geq 2$, and $n=2,$ $k=3^{2r-1}+3^r-3^{r-1}.$
Then $f(x)=x+\gamma \textup{Tr}_{2}(x^k)$ is a PP over $\mathbb{F}_{q^{2}}$, where $\gamma\in \mathbb{F}_{q^{2}}$ satisfying $(\gamma-1)^{\frac{q-1}{2}}=\gamma^{\frac{q-1}{2}}$.
\end{theorem}

\begin{proof}
We will show that $f(x)=a$ has a unique nonzero solution for each $a\in \mathbb{F}_{q^2}$. That is, for the equation
\begin{equation}\label{eq:pp1}
x+\gamma (x^k+\overline{x}^k)=a,
\end{equation}
there exists a unique  solution $x\in
\mathbb{F}_{q^2}.$
Raising both sides of Eq. \eqref{eq:pp1} to the $3$-th power, noting that $\overline{x}^q=x^{q^2}=x,$ we get
\begin{equation}\label{eq:pp2}
x^{3}+\gamma^3 (x\overline{x}^2+\overline{x}x^2)=a^3,
\end{equation}
Since $(\gamma-1)^{\frac{q-1}{2}}=\gamma^{\frac{q-1}{2}},$ we can easily obtain $\gamma \in \mathbb{F}_q.$ Raising both sides of Eq. \eqref{eq:pp2} to the $q$-th power
\begin{equation}\label{eq:pp3}
\overline{x}^{3}+\gamma^3 (\overline{x}x^2+x\overline{x}^2)=\overline{a}^3.
\end{equation}
By Eq. \eqref{eq:pp2} and Eq. \eqref{eq:pp3}, we have $(x-\overline{x})^3=(a-\overline{a})^3.$ Since $\textup{gcd}(3,q^2-1)=1,$ we get
\begin{equation}\label{eq:pp4}
\overline{x}=x+\overline{a}-a.
\end{equation}
Substituting Eq. \eqref{eq:pp4} into Eq. \eqref{eq:pp2}, after some simple computation, we obtain
\begin{equation}\label{eq:pp5}
x^3+\Big(\frac{\gamma}{1-\gamma}\Big)^3(\overline{a}-a)^2x=\Big(\frac{a}{1-\gamma}\Big)^3.
\end{equation}
Therefore, $x$ is a solution of Eq. \eqref{eq:pp1} if and only if it is a solution to following equations
\begin{equation}\label{eqs1}
\left\{
\begin{array}{rcl}
x^3+\Big(\frac{\gamma}{1-\gamma}\Big)^3(\overline{a}-a)^2x&=&\Big(\frac{a}{1-\gamma}\Big)^3, \\[1.0ex]
\overline{x}-x&=&\overline{a}-a.
\end{array}\right.
\end{equation}

Next, we will prove that Eq. \eqref{eqs1} has at most one solution. If it has two distinct solutions, denoted by $x_1,x_2.$ By $\overline{x_1}-x_1=\overline{a}-a=\overline{x_2}-x_2,$ we get $\overline{x_1-x_2}=x_1-x_2,$ that is $x_1-x_2\in \mathbb{F}_q.$ Write $c=x_1-x_2\in \mathbb{F}_{q}^*,$ we have $x_2+c, x_2, x_2-c$ are three solutions to the first equation of Eq. \eqref{eqs1}. Then, we obtain $c^2=\Big(\frac{\gamma}{\gamma-1}\Big)^3(\overline{a}-a)^2.$ Note that $Y=\{\gamma\in \mathbb{F}_{q^2} |(\gamma-1)^{\frac{q-1}{2}}=\gamma^{\frac{q-1}{2}}\}\subset \mathbb{F}_q^{*}.$ Let $z=\frac{\gamma}{\gamma-1},$ we have $Z=\{z \in \mathbb{F}_{q^2} | z^{\frac{q-1}{2}}=1\}=\langle \omega^{2(q+1)} \rangle \setminus\{1\}\subset \mathbb{F}_q^{*},$ where $w$ is a primitive element of $\mathbb{F}_{q^2}.$

Thus we have $c=\pm d(\overline{a}-a),$ where $d=w^{3j(q+1)} \in \mathbb{F}_q$ for some $j.$ Below, we consider the two cases where $c=d(\overline{a}-a)$ and $c=-d(\overline{a}-a)$.

{\bf Case 1: }
Raising both sides of $c=d(\overline{a}-a)$ to the $q$-th power, we have $\overline{c}=d(a-\overline{a})=-c,$ which leads to a contradiction since $c\in \mathbb{F}_{q}^*.$

{\bf Case 2: }
Raising both sides of $c=-d(\overline{a}-a)$ to the $q$-th power, we have $\overline{c}=-d(a-\overline{a})=-c,$ which leads to a contradiction since $c\in \mathbb{F}_{q}^*.$

This completes the proof.
\end{proof}

\begin{remark}
When $q=9,~ n=2,~ k=33,$ the condition on $\gamma$ in the above theorem is a little different from $\gamma^2-\gamma=1.$ Noting that $(\gamma-1)^4=\gamma^4$ implies $(\gamma+1)(\gamma^2-\gamma-1)=0.$ It follows that $\gamma=-1$ is also a suitable solution, which is contained in a class of PPs proposed in \cite{kyureghyan2016}.
\end{remark}

\section{Conclusion}\label{conclusion}
This paper demonstrates some new results on permutation polynomials. We prove two conjectures on permutation polynomial proposed recently by Wu and Li \cite{WuL}. Moreover, we give a new class of trinomial PPs of the form $x+\gamma \textup{Tr}_{n}(x^k)$, which generalizes two examples of \cite{kyureghyan2016}, namely, Examples \ref{ex5.2} and \ref{ex5.3}.

One question remains at the end of this paper: How to extend Examples \ref{ex5.1}, \ref{ex5.4} and \ref{ex5.5} to infinite classes? If it can be solved, we would have a complete understanding on PPs over fields $\mathbb{F}_{q^n}$ of the form $x+\gamma \textup{Tr}_{n}(x^k)$ with $\gamma\in \mathbb{F}_{q^n}^*,$ $n>1$ and $q^n<5000.$ Due to this, we propose the following problem:

\begin{problem}
Is it possible to extend Examples \ref{ex5.1}, \ref{ex5.4} and \ref{ex5.5} to infinite classes with the form $x+\gamma \textup{Tr}_{n}(x^k)$?
\end{problem}

\end{document}